\documentclass[11pt,a4paper]{article}
\usepackage{amssymb,amsmath,amsthm,multirow,xcolor,cancel,bbm,tikz,mathrsfs,calc,subfig}
\usepackage[textheight=25cm,textwidth=16cm]{geometry}

\newtheorem{definition}{Definition}
\newtheorem{proposition}{Proposition}
\newtheorem{lemma}{Lemma}
\newtheorem{theorem}{Theorem}

\newtheorem{remark}{Remark}
\newtheorem{example}{Example}

\newtheorem{claim}{Claim}
\newcommand{\ONE}{\mathbf{1}}
\newcommand{\ZERO}{\mathbf{0}}
\newcommand{\BF}[1]{{\bf\boldmath #1\unboldmath}}

\newcommand{\B}{\{0,1\}}

\newcommand{\wH}{w_\mathrm{H}}

\title{Fixing monotone Boolean networks asynchronously}

\author{
Julio Aracena\footnote{CI2MA and Departamento de Ingenier\'ia Matem\'atica, Universidad de Concepci\'on, Chile},
Maximilien Gadouleau\footnote{Department of Computer Science, Durham University, UK},
Adrien Richard\footnote{Laboratoire I3S, UMR CNRS 7271 \& Universit\'e C\^ote d'Azur, France},
Lilian Salinas\footnote{Department of Computer Sciences and CI2MA, University of Concepci\'on, Chile}
}

\date{January 30, 2018; revised August 30, 2019}

\begin{document}

\maketitle

\begin{abstract}
The asynchronous automaton associated with a Boolean network $f:\B^n\to\B^n$ is considered in many applications. It is the finite deterministic automaton with set of states $\B^n$, alphabet $\{1,\dots,n\}$, where the action of letter $i$ on a state $x$ consists in either switching the $i$th component if $f_i(x)\neq x_i$ or doing nothing otherwise. This action is extended to words in the natural way. We then say that a word $w$ {\em fixes} $f$ if, for all states $x$, the result of the action of $w$ on $x$ is a fixed point of $f$. In this paper, we ask for the existence of fixing words, and their minimal length. Firstly, our main results concern the minimal length of words that fix {\em monotone} networks.  We prove that, for $n$ sufficiently large, there exists a monotone network $f$ with $n$ components such that any word fixing $f$ has length $\Omega(n^2)$. For this first result we prove, using Baranyai's theorem, a property about shortest supersequences that could be of independent interest: there exists a set of permutations of $\{1,\dots,n\}$ of size $2^{o(n)}$ such that any sequence containing all these permutations as subsequences is of length $\Omega(n^2)$. Conversely, we construct a word of length $O(n^3)$ that fixes all monotone networks with $n$ components. Secondly, we refine and extend our results to different classes of fixable networks, including networks with an acyclic interaction graph, increasing networks, conjunctive networks, monotone networks whose interaction graphs are contained in a given graph, and balanced networks. 
\end{abstract}

\section{Introduction}

A \BF{Boolean network} (\BF{network} for short) is a finite dynamical system usually defined by a function  
\[
f:\B^n\to\B^n,\qquad x=(x_1,\dots,x_n)\mapsto f(x)=(f_1(x),\dots,f_n(x)).
\]

Boolean networks have many applications. In particular, since the seminal papers of McCulloch and Pitts \cite{MP43}, Hopfield \cite{H82}, Kauffman \cite{K69,K93} and Thomas \cite{T73,TA90}, they are omnipresent in the modeling of neural and gene networks (see \cite{B08,N15} for reviews). They are also essential tools in computer science, for the network coding problem in information theory \cite{ANLY00,GRF16} and memoryless computation \cite{BGT14, CFG14a, GR15b}. 

The ``network'' terminology comes from the fact that the \BF{interaction graph} of $f$ is often considered as the main parameter of $f$: it is the directed graph  with vertex set $[n]:=\{1,\dots,n\}$ and an edge from $j$ to $i$ if $f_i$ depends on $x_j$, that is, if there exist $x,y\in\B^n$ that only differ in the component $j$ such that $f_i(x)\neq f_i(y)$. An illustration is given in Figure~\ref{fig:f}(a-b).

From a dynamical point of view, the successive iterations of $f$ describe the so called \textit{synchronous dynamics}: if $x^t$ is the state of the system at time $t$, then $x^{t+1}=f(x^t)$ is the state of the system at the next time. Hence, all components are updated in parallel at each time step. However, when Boolean networks are used as models of natural systems, such as gene networks, synchronicity can be an issue. This led researchers to consider the \textit{(fully) asynchronous dynamics}, where one component is updated at each time step (see e.g. \cite{T91,TA90,TK01,A-J16} for Boolean networks, and \cite{F13,DFM12,DFMM13} for the closely related model of cellular automata). In our setting, given an infinite sequence $i_1, i_2 \dots$ of elements in $[n]$, called \textit{updating strategy}, and an initial state $x^0$, the resulting asynchronous dynamics is given by the recurrence $x^{t+1}=f^{i_{t+1}}(x^t)$, where, for each component $i\in [n]$ and state $x\in\B^n$,
\[
f^i(x):=(x_1,\dots,f_i(x),\dots,x_n). 
\] 

Functions $f^1,\dots,f^n$ define, in a natural way, a deterministic finite automaton called \BF{asynchronous automaton} of $f$: the set of states is $\B^n$, the alphabet is $[n]$ and $f^i(x)$ is the result of the action of a letter $i$ on a state $x$; see Figure~\ref{fig:f}(c) for an illustration. This action is extended to words on the alphabet $[n]$ in the natural way: the result of the action of a word $w=i_1i_2\dots i_k$ on a state $x$ is inductively defined by $f^w(x):=f^{i_2\dots i_k}(f^{i_1}(x))$ or, equivalently,  $f^w(x):=(f^{i_k}\circ f^{i_{k-1}}\circ\dots\circ f^{i_1})$. Hence, given an updating strategy  $i_1, i_2 \dots$ and an initial state $x^0$, the state of the system at time $t$ in the resulting asynchronous dynamics is $x^t=f^{i_1i_2\dots i_t}(x^0)$. 

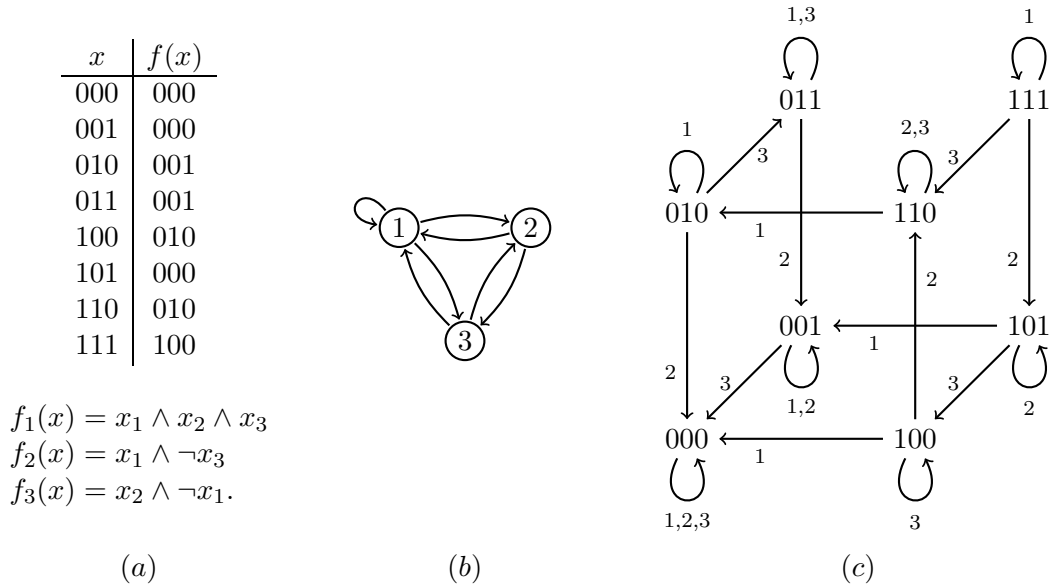
\begin{figure}
\[
\begin{array}{ccc}
\begin{array}{c}
\begin{array}{c|c}
	x & f(x)\\\hline
	000 & 000\\
	001 & 000\\
	010 & 001\\
	011 & 001\\
	100 & 010\\
	101 & 000\\
	110 & 010\\
	111 & 100\\
\end{array}
\\~\\
\begin{array}{l}
	f_1(x)= x_1 \land x_2 \land x_3\\
	f_2(x)= x_1 \land \neg x_3\\
	f_3(x)= x_2 \land \neg x_1.
\end{array}
\end{array}
&
\begin{array}{c}
\begin{tikzpicture}
\node[outer sep=1,inner sep=2,circle,draw,thick] (1) at (150:1){$1$};
\node[outer sep=1,inner sep=2,circle,draw,thick] (2) at (30:1){$2$};
\node[outer sep=1,inner sep=2,circle,draw,thick] (3) at (270:1){$3$};
\draw[->,thick] (1.128) .. controls (140:1.8) and (160:1.8) .. (1.172);
\draw[->,thick,white] (2.8) .. controls (20:1.8) and (40:1.8) .. (2.52);
\path[->,thick]
(1) edge[bend left=15] (2)
(2) edge[bend left=15] (1)
(2) edge[bend left=15] (3)
(3) edge[bend left=15] (2)
(3) edge[bend left=15] (1)
(1) edge[bend left=15] (3)
;
\end{tikzpicture}
\end{array}
&
\begin{array}{c}
\begin{tikzpicture}
\def\d{2}
\node (000) at (0,0){000};
\node at (0,-1.1){\scriptsize 1,2,3};
\node (001) at (1.5,1.5){001};
\node at (1.5,0.4){\scriptsize 1,2};
\node (010) at (0,3){010};
\node at (0,4.1){\scriptsize 1};
\node (011) at (1.5,4.5){011};
\node at (1.5,5.6){\scriptsize 1,3};
\node (100) at (3,0){100};
\node at (3,-1.1){\scriptsize 3};
\node (101) at (4.5,1.5){101};
\node at (4.5,0.4){\scriptsize 2};
\node (110) at (3,3){110};
\node at (3,4.1){\scriptsize 2,3};
\node (111) at (4.5,4.5){111};
\node at (4.5,5.6){\scriptsize 1};
\draw[->,thick] (000) .. controls ++(-0.5,-1) and ++(0.5,-1) .. (000);
\draw[->,thick] (001) .. controls ++(-0.5,-1) and ++(0.5,-1) .. (001);
\draw[->,thick] (100) .. controls ++(-0.5,-1) and ++(0.5,-1) .. (100);
\draw[->,thick] (101) .. controls ++(-0.5,-1) and ++(0.5,-1) .. (101);
\draw[->,thick] (010) .. controls ++(0.5,1) and ++(-0.5,1) .. (010);
\draw[->,thick] (011) .. controls ++(0.5,1) and ++(-0.5,1) .. (011);
\draw[->,thick] (110) .. controls ++(0.5,1) and ++(-0.5,1) .. (110);
\draw[->,thick]  (111) .. controls ++(0.5,1) and ++(-0.5,1) .. (111);
\path[thick,->,draw,black]
(001) edge node[near end, above]{\scriptsize 3} (000)
(010) edge node[near end, left]{\scriptsize 2} (000)
(010) edge node[near end, below]{\scriptsize 3} (011)
(011) edge node[near end, left]{\scriptsize 2} (001)
(100) edge node[near end, below]{\scriptsize 1} (000)
(100) edge node[near end, right]{\scriptsize 2} (110)
(101) edge node[near end, below]{\scriptsize 1} (001)
(101) edge node[near end, above]{\scriptsize 3} (100)
(110) edge node[near end, below]{\scriptsize 1} (010)
(111) edge node[near end, above]{\scriptsize 3} (110)
(111) edge node[near end, left]{\scriptsize 2} (101)
;
\end{tikzpicture}
\end{array}
\\
(a)&(b)&(c)
\end{array}
\]
\caption{\label{fig:f} (a) A network $f$ given under two different forms (a table and logical formulas). (b) The interaction graph of $f$. (c) The asynchronous automaton of $f$.}
\end{figure}

In this paper, we introduce and study the following concepts.

\begin{definition}
A \BF{fixed point} of $f$ is any state $x$ such that $f(x) = x$. A word $w$ over the alphabet $[n]$ \BF{fixes} $f$ if $f^w(x)$ is a fixed point of $f$ for every state $x$. If $w$ fixes $f$ then $w$ is a \BF{fixing word} for $f$. If $f$ admits a fixing word then $f$ is \BF{fixable}. If $f$ is fixable then the \BF{fixing length} of $f$, denoted $\lambda(f)$, is the length of a shortest word fixing $f$. 
\end{definition}

Hence, $w$ fixes $f$ if $w$ sends any state to a fixed point of $f$. This corresponds to the situation where the start state of the asynchronous automaton of $f$ is undetermined, and the accepting states are exactly the fixed points of $f$. As such, there is an obvious connection between fixing word and \textit{synchronizing words}: if $f$ has a unique fixed point, then $w$ is a synchronizing word for the asynchronous automaton of $f$ if and only if $w$ fixes $f$. For instance, the network in Figure~\ref{fig:f} is fixed by $w = 1231$. Hence it is fixable, and since it has a unique fixed point, its asynchronous automaton is synchronizing. 

It is easy to check that $f$ is fixable if and only if, for every state $x$, there is a word $w$ such that $f^w(x)$ is a fixed point. We may then think that fixability is a rather strong property. However it is not. Indeed, Bollob\'as, Gotsman and Shamir \cite{BGS93} showed that, considering the uniform distribution on the set of $n$-component networks, the probability for $f$ to be fixable when $f$ has at least one fixed point tends to $1$ as $n\to\infty$. Thus almost all networks with a fixed point are fixable. In turn, this shows that, for $n$ large, a positive fraction of all $n$-component networks are fixable. 

\begin{theorem}[Bollob\'as, Gotsman and Shamir \cite{BGS93}]\label{thm:huge}
Let $\phi(n)$ be the number of fixable $n$-component networks. We have 
\[
\lim_{n\to \infty} \frac{\phi(n)}{2^{n2^n}}=1-\frac{1}{e}.
\]
\end{theorem}


Thus the family of fixable networks is huge, and it makes sense to study it. In this paper, we focus on the fixing length of fixable networks, introduced above. For instance, we have seen that $w=1231$ fixes the network in Figure~\ref{fig:f}, thus this network has fixing length at most $4$, and it is easy to see that no word of length three fixes this network. Thus it has fixing length exactly~$4$. It is easy to construct a fixable $n$-component network $f$ such that $\lambda(f)$ is exponential in $n$, as follows. Let $x^1, \dots, x^{2^n}$ be a Gray code ordering of $\B^n$, i.e. $x^k$ and $x^{k+1}$ only differ by one component, say $i_k$, for $1\leq k<2^n$. Then let $f(x^k):= x^{k+1}$ for $1\leq k<2^n$ and $f(x^{2^n}) := x^{2^n}$. It is clear that $i_1,\dots,i_{2^n-1}$ is the unique shortest word fixing $f$, and thus $\lambda(f) = 2^n -1$.

We are then interested in networks $f$ which can be fixed in polynomial time, i.e. $\lambda(f)$ is bounded by a polynomial in $n$. More strongly, we extend our concepts to entire families $\mathcal{F}$ of $n$-component networks. We say that $\mathcal{F}$ is fixable if there is a word $w$ such that $w$ fixes $f$ for all $f\in\mathcal{F}$, which is clearly equivalent to: all the members of $\mathcal{F}$ are fixable. The fixing length $\lambda(\mathcal{F})$ is defined naturally as the length of a shortest word fixing $\mathcal{F}$. We are then interested in families $\mathcal{F}$ which can be fixed in polynomial time. We will identify several such families, and for each, we will derive an upper bound on $\lambda(\mathcal{F})$ and a lower bound on the maximum $\lambda(f)$ for $f \in \mathcal{F}$. Up to our knowledge, the only result of this kind was given in \cite{FGW83}, where it is shown that the word $12\dots n$ repeated $n(3n-1)$ times fixes any $n$-component symmetric threshold network with weights in $\{-1,0,1\}$. This family of threshold networks has thus a cubic fixing length. We could also mention somewhat less connected works concerning the minimal, maximal and average convergence time toward fixed points in the asynchronous setting for some specific fixable networks  \cite{A85,MRRS13,MRRS16,FTMS06,F13}. 

Our main results concern the family of monotone networks. We say that $f$ is \BF{monotone} if $x\leq y$ implies $f(x)\leq f(y)$ for all states $x,y$, where $x\leq y$ means $x_i\leq y_i$ for all $i\in [n]$. The fact that monotone networks are fixable is not obvious and proved in \cite{MRRS13}. Our first main result shows that some monotone networks have a quadratic fixing length. 

\begin{theorem}\label{thm:main1}
For every positive integer $n$, there exists an $n$-component monotone network with fixing length $\Omega(n^2)$. 
\end{theorem}

For the proof, we establish, using Baranyai's theorem, the following property about shortest supersequences that could be of independent interest (permutations of $[n]$ are regarded as ordered arrangements, and any sequence obtained by deleting some elements in a sequence $s$ is a {\em subsequence} of $s$).

\begin{theorem}\label{thm:main2}
For every positive integer $n$, there exists a set of permutations of $[n]$ of size $2^{o(n)}$ such that any sequence containing all these permutations as subsequences is of length $\Omega(n^2)$.
\end{theorem}

Theorem~\ref{thm:main1} trivially shows that the fixing length of the family of $n$-component monotone networks is at least quadratic, but we have not been able to obtain a super-quadratic lower-bound for this fixing length. Our second main result is that, conversely, the fixing length of the family of $n$-component monotone networks is at most cubic.

\begin{theorem}\label{thm:main3}
For every positive integer $n$, there is a word of length $O(n^3)$ that fixes every $n$-component monotone network. 
\end{theorem} 

Our last main result refines the previous one using the interaction graph. The \BF{transversal number} of a directed graph is the minimum number of vertices to delete to make the directed graph acyclic.

\begin{theorem}\label{thm:main4}
Let $G$ be an $n$-vertex directed graph with transversal number $\tau$. There is a word of length $O(\tau^2n)$ that fixes every $n$-component monotone network with an interaction graph isomorphic to a subgraph of $G$. 
\end{theorem} 

Note that Theorem~\ref{thm:main3} trivially shows that the fixing length of a given monotone network is at most cubic, and we have not been able to obtain a sub-cubic upper-bound for this fixing length. However, Theorem~\ref{thm:main4} shows that, for bounded transversal number, the fixing length of a given monotone network is only linear. Note also that one obtains Theorem~\ref{thm:main3} from Theorem~\ref{thm:main4} when $G$ is the complete directed graph on $n$ vertices (with $n^2$ edges). 

The paper is organized as follows. In Section \ref{sec:preliminaries}, we give some basic definitions and notations. We also introduce two families of networks, the acyclic and increasing networks, and show that the fixing length of these two families is asymptotically $n^2$. For that, we use results concerning $n$-complete words. The technics introduced are then used, in Section~\ref{sec:monotone}, to analyse the fixing length of monotone networks. Quantitative versions of Theorems~\ref{thm:main1}, \ref{thm:main2} and \ref{thm:main3} are proved there. Section~\ref{sec:extensions} gives some refinements and extensions. We first study the fixing length of conjunctive networks, which are particular monotone networks. Then, we prove a quantitative version of Theorem~\ref{thm:main4} and we study the fixing length of the family of balanced networks, which generalize monotone networks. Finally, a conclusion and some perspectives are given in Section~\ref{sec:conclusion}.

\section{Preliminaries}\label{sec:preliminaries}

\subsection{Basic definition and notations}

Let $w=w_1\dots w_p$ be a word. The length $p$ of $w$ is denoted $|w|$. If $S =\{i_1,i_2,\dots i_q\}\subseteq[p]$ with $i_1 <i_2 < \dots< i_q$, then we shall sometimes use the notation $w_S = w_{i_1}w_{i_2} \dots w_{i_q}$; if $S = \emptyset$, then $w_S := \epsilon$, where $\epsilon$ is the empty word. Any such $w_S$ is a \BF{subsequence} of $w$. Moreover, for any integers $a, b\in [p]$ we set $[a,b] = \{a, a+1, \dots, b\}$ and hence $w_{[a,b]}:=w_a,\dots,w_b$ if $a\leq b$ and $w_{[a,b]}:=\epsilon$ if $a>b$. Any such $w_{[a,b]}$ is a \BF{factor} of $w$. For any word $w$ and any $k \ge 1$, the word $k \cdot w$ is obtained by repeating $w$ exactly $k$ times; $0 \cdot w$ is the empty word. For all $i \in [n]$, we denote as $e_i$ the $i$-th unit vector, i.e. $e_i = (0, \dots, 0, 1, 0, \dots, 0)$ with the $1$ in position $i$. Given $x,y\in\B^n$, $x+y$ is applied componentwise and computed modulo two. For instance, $x$ and $x+e_i$ only differ in the $i$th position. The state containing only $1$s is denoted $\ONE$, and the state containing only $0$s is denoted $\ZERO$. We equip $\B^n$ with the partial order $\leq$ defined as follows: for all $x,y\in\B^n$, $x\leq y$ if and only if $x_i\leq y_i$ for all $i\in [n]$. The {\bf Hamming weight} of $x$, denote $\wH(x)$,  is the number of $1$s in $x$. Let $f$ be an $n$-component network. We set $f^\epsilon:=\mathrm{id}$ and, for any integer $i$ and $x\in\B^n$, we define $f^i(x)$ as in the introduction if $i\in [n]$, and $f^i(x):=x$ if $i\not\in [n]$. This extends the action of letters in $[n]$ to letters in $\mathbb{N}$, and by extension, this also defines the action of a word over the alphabet $\mathbb{N}$.

Graphs are always directed and may contain loops (edges from a vertex to itself). Paths and cycles are always directed and without repeated vertices. Given a graph $G$ with vertex set $V$ (such a graph is a graph \BF{on} $V$) and $I\subseteq V$, we denote by $G[I]$ the subgraph of $G$ induced by $I$, and $G\setminus I=G[V\setminus I]$. We refer the reader to the authoritative book on graphs by Bang-Jensen and Gutin \cite{BG08} for some basic concepts, notation and terminology.   

We now recall from the introduction the definition of monotone networks.

\begin{definition}[Monotone networks]
An $n$-component network $f$ is \BF{monotone} if, 
\[
\forall x,y\in\B^n,\qquad x\leq y\Rightarrow f(x)\leq f(y).
\]
The family of $n$-component monotone networks is denoted $F_M(n)$.
\end{definition}

The fixing length of $F_M(n)$ is denoted $\lambda_M(n)$. More generally, if $F_X(n)$ is any family of $n$-component fixable networks, then $\lambda_X(n)$ is the fixing length of $F_X(n)$. If $G$ is a graph on $[n]$, then $F(G)$ denotes the set of $n$-component networks $f$ such that the interaction graph of $f$ is a subgraph of $G$. Then, $F_X(G):=F_X(n)\cap F(G)$ and $\lambda_X(G)$ is the fixing length of $F_X(G)$.

\subsection{Acyclic networks} \label{sec:acyclic}

Our first results concern acyclic networks. 

\begin{definition}[Acyclic networks]
An $n$-component network $f$ is \BF{acyclic} if its interaction graph is acyclic. The family of $n$-component acyclic networks is denoted $F_A(n)$. 
\end{definition}

An important property of acyclic networks is that they have a unique fixed point \cite{R80} and that they have an acyclic asynchronous graph \cite{R95} (the asynchronous graph of an $n$-component network $f$ is the graph on $\B^n$ with an edge from $x$ to $y$ $f^i(x)=y\neq x$ for some $i\in [n]$). This obviously implies that $F_A(n)$ is fixable. We show here that the fixing length of acyclic networks are rather easy to understand. The techniques used will be useful later, for analyzing the fixing length of monotone networks. 

\begin{lemma}\label{lem:acyclic1}
Let $G$ be an acyclic graph on $[n]$ and $f\in F(G)$. If a word $w$ contains, as subsequence, a topological sort of $G$, then $w$ fixes $f$. Furthermore, $\lambda(f)=n$. 
\end{lemma}

\begin{proof}
Let $y$ be the unique fixed point of $f$. Let $u=i_1i_2 \dots i_n$ be topological sort of $G$, and let $w=w_1w_2\dots w_p$ be any word containing $u$ as subsequence. Hence, there is a increasing sequence of indices $j_1j_2\dots j_n$ such that $u=w_{j_1}w_{j_2}\dots w_{j_n}$. Let $x^0$ be any initial state, and for all $q\in [p]$, let $x^q$ be obtained from $x^{q-1}$ by updating $w_q$, that is, $x^q:=f^{w_q}(x^{q-1})$. Equivalently, $x^q:=f^{w_{[1,q]}}(x^0)$. Let us prove, by induction on $k\in [n]$, that $x^q_{i_k}=y_{i_k}$ for all $j_k\leq q\leq p$. Since $i_1$ is a source of the interaction graph, $f_{i_1}$ is a constant. Thus $f_{i_1}(x^q)=f_{i_1}(y)=y_{i_1}$ for all $q\in [p]$, and since $w_{j_1}=i_1$, we deduce that $x^q_{i_1}=y_{i_1}$ for all $j_1\leq q\leq p$. Let $1<k\leq n$. Since $f_{i_k}$ only depends on components $i_l$ with $1\leq l<k$, and since, by induction, $x^q_{i_l}=y_{i_l}$ for all $1\leq l<k$ and $j_{k-1}\leq q\leq p$, we have $f_{i_k}(x^q)=f_{i_k}(y)=y_{i_k}$ for all $j_{k-1}\leq q\leq p$. Since $w_{j_k}=i_k$ we deduce that $x^q_{i_k}=y_{i_k}$ for all $j_k\leq q\leq p$, completing the induction step. Hence, $f^w(x^0)=x^q=y$ for any initial state $x^0$, thus $w$ fixes $f$. 

We deduce that, in particular, any topological sort $u$ of $G$ fixes $f$, thus $\lambda(f)\leq n$. Conversely, if a word $w$ fixes $f$, then $f^w(\neg y) = y$, and hence at least $n$ asynchronous updates are required, that is, the length of $w$ is at least $n$. Thus $\lambda(f)=n$.
\end{proof}

The converse of the previous proposition is false in general (for instance if $f$ is the $3$-component network defined by $f_1(x)=0$, $f_2(x)=x_1$ and $f_3(x)=x_1\land x_2$, then $132$ fixes $f$ while $123$ is the unique topological sort of the interaction graph of $f$) but it holds for conjunctive networks, which are specific monotone networks.

\begin{definition}[Conjunctive networks]
An $n$-component network $f$ is \BF{conjunctive} if, for all $i\in [n]$, there exists $J_i\subseteq [n]$ such that,
\begin{equation}\label{eq:conj}
\forall x\in\B^n,\qquad f_i(x)= \bigwedge_{j\in J_i}x_j,
\end{equation}
and $f_i(x)=1$ is $J_i$ is empty. The family of $n$-component conjunctive networks is denoted $F_C(n)$. Let $G$ be a graph on $[n]$. The \BF{conjunctive network on $G$} is the unique conjunctive network whose interaction graph is $G$. Namely, it is the $n$-component network $f$ such that \eqref{eq:conj} holds for all $i\in[n]$ when $J_i$ is the set of in-neighbors of $i$ in $G$. 
\end{definition}

\begin{lemma}\label{lem:acyclic2}
Let $G$ be an acyclic graph on $[n]$ and let $f$ be the conjunctive network on $G$. A word $w$ fixes $f$ if and only if it contains, as subsequence, a topological sort of $G$.
\end{lemma}

\begin{proof}
According to Lemma~\ref{lem:acyclic1}, it is sufficient to prove that if $w=w_1w_2\dots w_p$ fixes $f$ then $w$ contains, as subsequence, a topological sort of $G$. Let $x^0:=\ZERO$ and $x^q=f^{w_q}(x^{q-1})$ for all $q\in [p]$. Since $w$ fixes $f$ and since $\ONE$ is the unique fixed point of $f$, we have $f^w(x^0)=x^p=\ONE$. Thus for each $i\in [n]$, there exists $t_i$ such that $x^{t_i}_i=1$ and $x^q_i=0$ for all $0\leq q<t_i$. We have, obviously, $w_{t_i}=i$. Let $i_1i_2\dots i_n$ be the enumeration of the vertices of $G$ such that $t_{i_1}t_{i_2}\dots t_{i_n}$ is increasing. In this way $i_1i_2\dots i_n$ is a subsequence of $w$, and it follows the topological order. Indeed, suppose that $G$ has an edge from $i_k$ to $i_l$. Since $f_{i_l}(x^{t_{i_l}-1})=x^{t_{i_l}}_{i_l}=1$, we have $x^{t_{i_l}-1}_{i_k}=1$, and thus $t_{i_k}<t_{i_l}$, that is, $i_k$ is before $i_l$ in the enumeration.
\end{proof}

As an immediate application we get the following characterization. 

\begin{proposition}\label{pro:acyclic3}
Let $G$ be an acyclic graph on $[n]$. A word $w$ fixes $F(G)$ if and only if it contains, as subsequence, a topological sort of $G$. 
\end{proposition}

To go further, we need the following concepts. 

\begin{definition}[Complete word] 
A word $w$ is {\bf complete} for a finite set $S$ (or $S$-complete) if it contains, as subsequence, all the permutations of $S$. An \BF{$n$-complete word} is a $[n]$-complete word. The length of a shortest $n$-complete word is denoted $\lambda(n)$.
\end{definition}

Interestingly, $\lambda(n)$ is unknown. Let $w^1,\dots,w^n$ be $n$ permutations of $[n]$ (not necessarily distinct). Then the concatenation $w^1w^2\dots w^n$ clearly contains all the permutations of $[n]$. Thus $\lambda(n)\leq n^2$. Conversely, if $w$ contains all the permutations of $n$, then ${|w|\choose n}$ is at least $n!$ and we deduce that $|w|\geq n^2/e^2$ (this simple counting argument will be reused later). This shows that the magnitude of $\lambda(n)$ is quadratic. We have however the following tighter bounds.

\begin{theorem} We have $\lambda(n)\sim n^2$. More precisely:
\[
\begin{array}{lll}
\lambda(n)\leq n^2-2n+4&\text{for all }n\geq 1& \text{\em \cite{A74}}\\[1mm]
\lambda(n)\leq n^2-2n+3&\text{for all }n\geq 10&\text{\em \cite{Z11}}\\[1mm]
\lambda(n) \le \left\lceil n^2-\frac{7}{3}n+\frac{19}{3} \right\rceil&\text{for all }n\geq 7& \text{\em\cite{R12}}\\[1mm]
\lambda(n) \ge n^2-C_\varepsilon n^{7/4+\varepsilon}&&\text{\em\cite{KK76}}
\end{array}
\]
where $\varepsilon>0$ and $C_\varepsilon$ is a positive constant that only depends on $\varepsilon$. 
\end{theorem}

We also need another family of networks.

\begin{definition}[Path networks]
An $n$-component network $f$ is a \BF{path network} if its interaction graph if a path. The family of $n$-component path networks is denoted $F_P(n)$. 
\end{definition}

Note that path networks are both acyclic and conjunctive. Note also that an $n$-component network $f$ is a path network if and only if there is a permutation $i_1i_2\dots i_n$ of $[n]$ such that $f_{i_1}(x)=1$ and $f_{i_k}(x)=x_{i_{k-1}}$ for all $1<k\leq n$ and $x\in\{0,1\}^n$. There is thus a natural bijection between the permutations of $[n]$ and $F_P(n)$. We show below that the family $F_P(n)$ has a quadratic fixing length. 

\begin{lemma}\label{lem:acyclic3}
A word $w$ fixes $F_P(n)$ if and only if it is $n$-complete. Hence $\lambda_P(n) = \lambda(n)$.
\end{lemma}

\begin{proof}
By Lemma~\ref{lem:acyclic1}, any $n$-complete word fixes $F_A(n)$ and thus $F_P(n)$ in particular. Conversely, suppose that $w$ fixes $F_P(n)$. Since each permutation of $n$ is the unique topological sort of the interaction graph of exactly one network in $F_P(n)$, by Lemma \ref{lem:acyclic2}, $w$ contains, as subsequence, the $n!$ permutations of $n$. Thus $w$ is $n$-complete. 
\end{proof}

As an immediate consequence, we get the following proposition. 

\begin{proposition} \label{prop:FA}
A word $w$ fixes $F_A(n)$ if and only if it is $n$-complete. Hence $\lambda_A(n) = \lambda(n)$.
\end{proposition}


\begin{remark}
By Lemma \ref{lem:acyclic3} and Proposition \ref{prop:FA}, it is as hard to fix $F_P(n)$ as to fix $F_A(n)$: these two families have the same quadratic fixing length, while $F_P(n)$ is much smaller than $F_A(n)$ (the former has $n!$ members while the latter has $2^{\Theta(2^n)}$ members). We shall use this to our advantage when designing a monotone network with quadratic fixing length in Section \ref{sec:one_monotone}.
\end{remark}

\subsection{Increasing networks}

\begin{definition}[Increasing networks]
An $n$-component network $f$ is \BF{increasing} if, 
\[
\forall x\in\B^n,\qquad x\leq f(x).
\]
The family of $n$-component increasing networks is denoted $F_I(n)$.
\end{definition}

We prove below that, as for path networks and acyclic networks, the fixing length of increasing networks is $\lambda(n)$. Increasing networks are thus relatively easy to fix collectively. We shall use this fact when constructing a cubic word fixing all monotone networks in Section \ref{sec:all_monotone}.

\begin{lemma}\label{lem:increasing1}
Let $f$ be an $n$-component network and $x\in\{0,1\}^n$. If $f^u(x)\leq f^{uv}(x)$ for any words $u$ and $v$, then $f^w(x)$ is a fixed point of $f$ for any word $w$ containing all the permutations of $\{i:x_i=0\}$. Similarly, if $f^u(x)\geq f^{uv}(x)$ for any words $u$ and $v$, then $f^w(x)$ is a fixed point of $f$ for any word $w$ containing all the permutations of $\{i:x_i=1\}$. 
\end{lemma}

\begin{proof}
Suppose that $f^u(x)\leq f^{uv}(x)$ for any words $u$ and $v$, and that $w=w_1w_2\dots w_p$ be $S$-complete, with $S:=\{i:x_i=0\}$. Let $x^0:=x$ and $x^q:=f^{w_q}(x^{q-1})$ for all $q\in [p]$. By hypothesis, $x^0\leq x^1\leq\dots \leq x^q$. Suppose for the sake of contradiction that $x^p=f^w(x)$ is not a fixed point, i.e. there is $j$ such that $f_j(x^p)> x^p_j$. Let $t_1,\dots,t_m\in [p]$ be the set of positions such that $x^{t_k-1}<x^{t_k}$, and let $i_k:=w_{t_k}$ for all $k\in [m]$. Clearly, $j\neq i_k$ for every $k\in [m]$. Setting $t_0:=0$, we have $x^{t_{k-1}}=x^{t_k-1}$, thus $f^{i_k}(x^q)=f^{i_k}(x^{t_k-1})=x^{t_k}>x^{t_k-1}=x^q$ for all $t_{k-1}\leq q<t_k$. We deduce that $i_k$ does not appear in $w_{[t_{k-1},t_k-1]}$ or, equivalently, $t_k$ is the first position of $i_k$ in $w_{[t_{k-1},p]}$. Similarly, we have $x^{t_m}=x^p$, thus $f_j(x^q)=f_j(x^p)>x^p_j=x^q_j$ for all $t_m\leq q\leq p$ and we deduce that $j$ does not appear in $w_{[t_m,p]}$.  Since $w$ is $S$-complete, the sequence $i_1i_2 \dots i_m j$ appears in $w$, say at positions $w_{s_1}w_{s_2}\dots w_{s_m}w_{s_{m+1}}$. Since $t_k$ is the first position of $i_k$ in $w_{[t_{k-1},p]}$, we have $s_k\geq t_k$ for all $k\in [m]$. In particular, $s_m\geq t_m$, thus $j$ appears in $w_{[t_m,p]}$ which is the desired contradiction. If $f^u(x)\geq f^{uv}(x)$ for any words $u$ and $v$ the proof is similar. 
\end{proof}

\begin{proposition}\label{pro:increasing2}
	A word $w$ fixes $F_I(n)$ if and only if it is $n$-complete. Hence $\lambda_I(n) = \lambda(n)$.
\end{proposition}

\begin{proof}
If $w$ is $n$-complete, then $w$ fixes $f$ by Lemma~\ref{lem:increasing1}. Conversely, suppose that $w$ fixes all $n$-component increasing networks and let $i_1i_2\dots i_n$ be any permutation of $[n]$. Let $y^0:=\ZERO$ and  $y^k := y^{k-1} + e_{i_k}$ for all $k\in [n]$. Then $y^0 y^1 \dots y^n$ is a chain from $\ZERO$ to $\ONE$ in the hypercube $Q_n$. Let $f$ be the $n$-component increasing network defined by
	$$
	f(x) := \begin{cases}
	y^{k+1} &\text{if } x = y^k\text{ and $1\leq k<n$},\\
	x &\text{otherwise}.
	\end{cases}
	$$
Then $\ONE$ is the unique fixed point of $f$ reachable from $\ZERO$ in the asynchronous graph, and it is easy to check that $f^w(\ZERO)=\ONE$ if and only if $i_1i_2\dots i_n$ is a subsequence of $w$. Thus $w$ is $n$-complete. 
\end{proof}

On the other hand, some increasing networks have quadratic fixing length. The proof requires the machinery developed for monotone networks, and as such we delay its proof until Section~\ref{sec:one_monotone}. 

\begin{theorem} \label{th:lambda(f)increasing}
For any $\varepsilon > 0$ and $n$ sufficiently large, there exists $f \in F_I(n)$ such that 
\[
\lambda(f) \ge \left(\frac{1}{e} - \varepsilon\right) n^2.
\]
\end{theorem}

\begin{remark}
A simple exercise shows that the number of increasing networks is doubly exponential: $|F_I(n)| = 2^{n2^{n-1}}$. It is then remarkable that, while some increasing networks have quadratic fixing length, all the increasing networks can be fixed together in quadratic time still. 
\end{remark}

\begin{remark}
The dual $\tilde{f}$ of a network $f$ is defined as $\tilde{f}(x) = f(x+\ONE)+\ONE$. It is easily checked that a word fixes $f$ if and only if it fixes $\tilde{f}$. Since a network is increasing if and only if its dual is decreasing, i.e. $x \ge f(x)$ for all $x$, the above results also holds for decreasing networks.
\end{remark}

\section{Monotone networks} \label{sec:monotone}
 
\subsection{A monotone network with quadratic fixing length} \label{sec:one_monotone}

The aim of this section is to exhibit a monotone network with quadratic fixing length. As we saw in Section \ref{sec:acyclic}, the family of path networks $F_P(n)$ has quadratic fixing length. Therefore, our strategy is to ``pack'' many of these path networks in the same network $f$. As an illustration of this strategy, we first describe a monotone network with fixing length of order $(n/\log n)^2$.

Let $n = m + r$ where $m! \le \binom{r}{r/2}$, and let us write $F_P(m)=\{h^1,\dots,h^{m!}\}$. There is then a surjection $\phi:X\to [m!]$ where $X$  is the set of states in $\{0,1\}^r$ with Hamming weight $r/2$. The $n$-component network $f$ then views the $r$ last components as controls, that decide, through $\phi$, which network in $F_P(m)$ to choose on the first $m$ components. More precisely, by identifying $\{0,1\}^n$ with $\{0,1\}^m\times \{0,1\}^r$, we define $f$ as follows:
$$
	f(x,y) := \begin{cases}
		(\ONE,y) &\text{if } \wH(y) > r/2,\\
		(h^{\phi(y)}(x),y)&\text{if }\wH(y) = r/2,\\
		(\ZERO,y) &\text{if }\wH(y) < r/2.
	\end{cases}
$$
The first and third cases are there to guarantee that $f$ is indeed monotone. Since any network in $F_P(m)$ can appear, a word fixing $f$ must fix $F_P(m)$. Thus a word fixing $f$ is $m$-complete, and hence has length $\Omega(m^2)$. Choosing $m = \Omega(n/ \log n)$ then yields $\Omega(n^2/ \log^2 n)$.

The network above reached a fixing length of $\Omega(m^2)$ because it packed all possible networks in $F_P(m)$. However, it did not reach quadratic fixing length because $m$ had to be $o(n)$ in order to embed all $m!$ networks of $F_P(m)$ in $X$. Thus, we show below that only a subexponential subset of $F_P(m)$ is required to guarantee $\Omega(m^2)$. This is equivalent to prove that there exists a subexponential set of permutations of $[m]$ such that any word containing these permutations as subsequences is of length $\Omega(m^2)$. In that case, we can use $m = (1-o(1))n$, and hence reach a fixing length of $\Omega(n^2)$. 

The main tool is Baranyai's theorem, see \cite{vW01}. 

\begin{theorem}[Baranyai]
If $a$ divides $n$, then there exists a collection of ${n\choose a}\frac{a}{n}$ partitions of $[n]$ into $\frac{n}{a}$ sets of size $a$ such that each $a$-subset of $[n]$ appears in exactly one partition.
\end{theorem}

\begin{lemma}
	Let $a$ and $b$ be positive integers, and $n=ab$. There exists a set of $a!{n\choose a}\leq n^a$ permutations of $[n]$ such any word containing all these permutations as subsequences is of length at least 
	\[
	\left(n^{-\frac{2b}{a}}\right)\frac{n(n-a)}{e}.
	\]
\end{lemma}

\begin{proof}
According to Baranyai's theorem, there exists a collection of $r:=b^{-1}{n\choose a}$ partitions of $[n]$ into $b$ sets of size $a$, such that each $a$-subset of $[n]$ appears in exactly one partition. Let $A^0,\dots,A^{r-1}$ be these partitions. For each $0\leq i<r$, we set 
\[
A^i=\{A^i_0,\dots,A^i_{b-1}\}.
\]
Then, for all $0\leq i<r$ and $0\leq j,k<b$ we set $S^{i,j}_k := A^i_{j+k}$ and
	\[
	S^{i,j} := S^{i,j}_0S^{i,j}_1\dots S^{i,j}_{b-1} = A^i_{j+0}A^i_{j+1}\dots A^i_{j+b-1}
	\]
	where addition is modulo $b$. So, the $S^{i,j}$ form a set of ${n\choose a}$ {\em ordered} partitions of $[n]$ in $b$ sets of size $a$. The interesting point is that, for all fixed $i$ and fixed $\ell$, the sequence $S^{i,0}_\ell S^{i,1}_\ell\dots S^{i,b-1}_\ell$ is a permutation of $A^i$ (namely $S^{i,\ell}$). Since each $a$-subset of $[n]$ appears in exactly one $A^i$, we deduce that, for any fixed $\ell$, the set of $S^{i,j}_\ell$ is {\em exactly} the set of $a$-subsets of $[n]$.

	Given an $a$-subset $X$ of $[n]$ and a permutation $\sigma$ of $[a]$, we set $\sigma(X)=i_{\sigma(1)}i_{\sigma(2)}\dots i_{\sigma(a)}$, where $i_1,i_2,\dots,i_a$ is an enumeration of the elements of $X$ in the increasing order. Let $\sigma^0,\dots,\sigma^{a!-1}$ be an enumeration of the permutations of $[a]$. For all $0\leq i<r$, $0\leq j<b$, $0\leq k<a!$, we set 
	\[
	\pi^{i,j,k}:=\sigma^k(S^{i,j}_0)\dots\sigma^k(S^{i,j}_{b-1}). 
	\]
	The $\pi^{i,j,k}$ form a collection of $a!{n\choose a}$ permutations of $[n]$. The interesting property is that, for $\ell$ fixed, the set of $\sigma^k(S^{i,j}_\ell)$ is {\em exactly} the set of words in $[n]^a$ without repetition, simply because, for $\ell$ fixed, the set of $S^{i,j}_\ell$ is exactly the set of $a$-subsets of $[n]$, as mentioned above. In particular, for $\ell$ fixed, the $\sigma^k(S^{i,j}_\ell)$ are pairwise distinct.

	Let $w=w_1w_2\dots w_p$ be a shortest word containing all the permutations $\pi^{i,j,k}$ as subsequences. We know that $|w|\leq\lambda(n)\leq n^2$. Let 
	\[
	\gamma^{i,j,k}:=\gamma^{i,j,k}_0\gamma^{i,j,k}_1\dots\gamma^{i,j,k}_b
	\]
	be the {\em profile} of $\pi^{i,j,k}$, defined recursively as follows: $\gamma^{i,j,k}_0:=0$ and, for all $0\leq \ell <b$, $\gamma^{i,j,k}_{\ell+1}$ is the smallest integer such that $\sigma^k(S^{i,j}_\ell)$ is a subsequence of the factor 
	\[
	w_{[\gamma^{i,j,k}_\ell+1,\gamma^{i,j,k}_{\ell+1}]}.
	\]
Since $\gamma^{i,j,k}_0=0$ and $1\leq \gamma^{i,j,k}_\ell\leq n^2$ for all $1\leq \ell\leq b$, there are at most $n^{2b}$ possible profiles. Thus there exist at least 
	\[
	s\geq \frac{a!{n\choose a}}{n^{2b}}
	\]
	permutations $\pi^{i,j,k}$ with the same profile. Let $\pi^{i_1,j_1,k_1},\dots,\pi^{i_s,j_s,k_s}$ be these permutations, and let $\gamma=(\gamma_0,\gamma_1,\dots,\gamma_b)$ be their profile. For all $0\leq \ell<b$, let 
	\[
	w^\ell:=w_{[\gamma_\ell+1,\gamma_{\ell+1}]}.
	\]
	By construction, $w^\ell$ contains, as subsequences, each of $\sigma^{k_1}(S^{i_1,j_1}_\ell),\dots,\sigma^{k_s}(S^{i_s,j_s}_\ell)$. Since these $s$ elements of $[n]^a$ are pairwise distinct (because, for fixed $\ell$, all the $\sigma^k(S^{i,j}_\ell)$ are pairwise distinct), this means that $w^\ell$ contains at least $s$ distinct subsequences of length $a$, and thus 
	\[
	{|w^\ell|\choose a}\geq s. 
	\] 
	We deduce 
	\[
	\frac{|w^\ell|^a}{a!}\geq {|w^\ell|\choose a}\geq s\geq \frac{a!{n\choose a}}{n^{2b}}\geq \frac{(n-a)^a}{n^{2b}}
	\]
	and thus
	\[
	|w^\ell|^a\geq a!\frac{(n-a)^a}{n^{2b}}\geq \left(\frac{a}{e}\right)^a\frac{(n-a)^a}{n^{2b}}\geq \left[\frac{a(n-a)}{en^{\frac{2b}{a}}}\right]^a.
	\]
	Consequently, 
	\[
	|w|\geq \sum_{0\leq\ell<b} |w^\ell|\geq  b\cdot \frac{a(n-a)}{en^{\frac{2b}{a}}}=
	\left(n^{-\frac{2b}{a}}\right)\frac{n(n-a)}{e}.
	\]
	
\end{proof}

We are now in position to prove that there is a subexponential set of permutations that requires a quadratic length to be represented in a supersequence. This is a quantitative version of Theorem~\ref{thm:main2} stated in the introduction. 

\begin{theorem}\label{th:resupersequence}
For any $\varepsilon>0$ and $n$ sufficiently large, there is a set of at most $n^{n^{\frac{1}{2}+\varepsilon}}$ permutations~of~$[n]$ such that any word containing all these permutations as subsequences is of length at least $(\frac{1}{e}-\varepsilon)n^2$. 
\end{theorem}

\begin{proof}
	Let $\varepsilon>0$ be arbitrarily small. Let $n$ be a positive integer, $a:=\lfloor n^{\frac{1}{2}+\varepsilon}\rfloor$,  $b:=\lfloor n^{\frac{1}{2}-\varepsilon}\rfloor$ and $m:=ab$. By the preceding lemma, there exist $s:=a!{m\choose a}\leq m^a\leq n^{n^{\frac{1}{2}+\varepsilon}}$ permutations $\pi^1,\dots,\pi^s$ of $[m]$ such that if $w$ is any word containing all the $\pi^i$ as subsequences then 
	\[
	|w|\geq 
	\left(m^{-\frac{2b}{a}}\right)\frac{m(m-a)}{e}\geq 
	\left(n^{-\frac{2b}{a}}\right)\frac{m(m-a)}{e}.
	\]
	First, $n^{-2ba^{-1}}\to 1$ as $n\to\infty$. Second, since $m=n+o(n)$, we have $m(m-a)=n^2-o(n^2)$. From these two observations we deduce that if $n$ is at least some constant $n_0$ that only depends on $\varepsilon$  then  
	\[
	|w|\geq(1-\varepsilon)\frac{n^2}{e}\geq (\frac{1}{e}-\varepsilon)n^2.
	\]
	For each $i\in [\ell]$, let $\tilde \pi^i$ be a permutation of $[n]$ that contains $\pi^{i}$ as a  subsequence. Then any word $\tilde w$ containing all the $\tilde \pi^i$ also contains all the $\pi^i$, so that $|\tilde w|\geq (\frac{1}{e}-\varepsilon)n^2$ if $n\geq n_0$. 
\end{proof}

Implementing the ``packing'' strategy described above, we obtain a monotone network with quadratic fixing length. This is a quantitative version of Theorem~\ref{thm:main1} stated in the introduction. 

\begin{theorem} \label{th:lambda(f)}
For any $\varepsilon > 0$ and $n$ sufficiently large, there exists $f \in F_M(n)$ such that 
\[
\lambda(f) \ge \left(\frac{1}{e} - \varepsilon\right) n^2.
\]
\end{theorem}

\begin{proof}
Let $\varepsilon>0$, let $n$ be a positive integer, and let $m$ be the largest integer such that 
\begin{equation}\label{eq:1}
m^{m^{\frac{1}{2}+\delta}} \le \binom{n-m}{\lfloor\frac{n-m}{2}\rfloor}, 
\end{equation}
with $\delta=\varepsilon/2$. Then $m=(1-o(1))n$ and thus if $n$ is large enough, then 
\[
m^2> n^2(1-\delta)
\]
and, according to Theorem~\ref{th:resupersequence}, there exists a collection $\pi^1,\dots,\pi^p$ of $p\leq m^{m^{\frac{1}{2}+\delta}}$ permutations of $[m]$ such that any word containing all these permutations as subsequences is of length at least $(\frac{1}{e}-\delta)m^2$. Let us regard these $p$ permutations as $m$-vertex paths, and let $h^1,\dots,h^p$ be the corresponding path networks. In other words, writing $\pi^k=\pi^k_1\pi^k_2\dots\pi^k_n$, we have, for all $x\in\{0,1\}^m$,   
\[
h^k_{\pi^k_1}(x)=1\text{ and }h^k_{\pi^k_l}(x)=x_{\pi^k_{l-1}}\text{ for all }1<l\leq m.
\] 
According to Lemma \ref{lem:acyclic2}, if a word $w$ fixes all the networks $h^k$ then it contains all the permutations $\pi^k$ as subsequences, and thus $|w|\geq (\frac{1}{e}-\delta)m^2$. 

Let $r:=n-m$. Then according to \eqref{eq:1} there is a surjection $\phi$ from the set of states in $\{0,1\}^r$ with Hamming weight $\lfloor\frac{r}{2}\rfloor$ to $[p]$. By identifying $\{0,1\}^n$ with $\{0,1\}^m\times\{0,1\}^r$, we then define the $n$-component network $f$ as follows:
$$
	f(x,y) = \begin{cases}
		(\ONE,y) &\text{if } \wH(y) > \lfloor r/2\rfloor,\\
		(h^{\phi(y)}(x),y)&\text{if }\wH(y) = \lfloor r/2\rfloor,\\
		(\ZERO,y) &\text{if }\wH(y) < \lfloor r/2\rfloor.
	\end{cases}
$$

Let us check that $f$ is monotone. Suppose $(x,y)\leq (x',y')$. If  $\wH(y)<\wH(y')$ we easily check that $f(x,y)\leq f(x',y')$. Otherwise, we have $y=y'$ and thus $\phi(y)=\phi(y')=k$ for some $k\in [p]$, and, since $h^k$ is monotone, we obtain $f(x,y)=(h^k(x),y)\leq (h^k(x'),y')=f(x',y')$. 

Let $w$ be any shortest word fixing $f$. Then it is clear that fixes $h^k$ for all $k\in [p]$. Thus 
$$
	|w| \ge (\frac{1}{e}-\delta)m^2 >(\frac{1}{e}-\delta)(1-\delta)n^2> \left( \frac{1}{e} - \varepsilon \right) n^2.
$$
\end{proof}

A similar argument works for increasing networks as well.

\begin{proof}[Proof of Theorem \ref{th:lambda(f)increasing}]
We use the same setup as the proof of Theorem \ref{th:lambda(f)}, excepted that the $m$-component networks $h^k$ and the $n$-component network $f$ are defined as follows. Let $k\in [p]$. We set $y^{k,0}:=\ZERO$ and  $y^{k,l} := y^{k,l-1} + e_{\pi^k_l}$ for all $l\in [m]$. We then define the $m$-component increasing network $h^k$ by
	$$
	h^k(x) := \begin{cases}
	y^{k,l+1} &\text{if } x = y^{k,l}\text{ and $1\leq l<m$},\\
	x &\text{otherwise}.
	\end{cases}
	$$
Then, as already said in the proof of Proposition~\ref{pro:increasing2}, a word fixes $h^k$ if and only if it contains $\pi^k$ as subsequence. Thus if a word $w$ fixes all the networks $h^k$ then it contains all the permutations $\pi^k$ as subsequences, and thus $|w|\geq (\frac{1}{e}-\delta)m^2$.  

Next, we define the $n$-component network $f$ as follows:
$$
	f(x,y) = \begin{cases}
		(\ONE,y) &\text{if } \wH(y) > \lfloor r/2\rfloor,\\
		(h^{\phi(y)}(x),y)&\text{if }\wH(y) = \lfloor r/2\rfloor,\\
		(\ONE,y) &\text{if }\wH(y) < \lfloor r/2\rfloor.
	\end{cases}
$$
We easily check that $f$ is increasing and that $w$ fixes $f$ if and only if it fixes $h^k$ for all $k\in [p]$. We then deduce as above that any word fixing $f$ is of length at least $\left(\frac{1}{e} - \varepsilon \right) n^2$.
\end{proof}

\subsection{Cubic word fixing all monotone networks} \label{sec:all_monotone}

What about the fixing length $\lambda_M(n)$ of the whole family $F_M(n)$ of $n$-component monotone networks? We have shown that some members have quadratic fixing length, namely $(\frac{1}{e}-\varepsilon) n^2$, and thus, obviously, $\lambda_M(n)\geq (\frac{1}{e}-\varepsilon)n^2$. But we can say something slightly better: we have shown that the family of path networks $F_P(n)$ has fixing length $\lambda(n)$, and since $F_P(n)\subseteq F_M(n)$ we obtain:
\[
\lambda_M(n)\geq \lambda(n). 
\]
We have no better lower-bound. Maybe the family of $n$-component conjunctive networks whose interactions graphs are disjoint union of cycles has fixing length greater than $\lambda(n)$ (this family can be equivalently defined as the set of monotone isometries of the hypercube $Q_n$). 

Concerning upper-bounds, we show below that $\lambda_M(n)$ is at most cubic, and this is the best upper-bound we have on the maximum fixing length of a member of $F_M(n)$. For that we construct inductively a word $W^n$ of cubic length that fixes $F_M(n)$.

\begin{definition}[Fixing word for monotone networks]
Let $W^1:=1$ and, for $n\geq 1$, let 
\[
W^{n+1} := W^n, n+1, \omega^n,
\]
where $\omega^n$ is a shortest $n$-complete word (of length $\lambda(n)$).
\end{definition}

\begin{example}
\begin{align*}
	W^2 &= 1, 2, 1\\
	W^3 &= 121, 3, 121\\
	W^4 &= 1213121, 4, 1213121\\
	W^5 &= 1213121 4 1213121, 5, 123412314213.
\end{align*}
\end{example}

\begin{theorem}\label{thm:monotone_universal}
The word $W^n$ fixes $F_M(n)$ for every $n\geq 1$. Therefore,  
\[
\lambda_M(n)\leq n+\sum_{i=1}^{n-1} \lambda(i)\leq \frac{n^3}{3}-\frac{3n^2}{2}+\frac{37n}{6}.
\]
\end{theorem}

This is a quantitative version of Theorem~\ref{thm:main3} stated in the introduction. The main idea is that, once the components $1$ to $n-1$ have been fixed, a monotone network behaves just like an increasing (or decreasing) network. Therefore, the network can be fixed in quadratic time from that point. 

\begin{lemma}\label{lm:local_increasing_or_decreasing}
Let $f$ be an $n$-component monotone network. If $x\leq f(x)$ then $f^u(x)\leq f^{uv}(x)$ for any words $u$ and $v$. Similarly, if $x\geq f(x)$ then $f^u(x)\geq f^{uv}(x)$ for any words $u$ and $v$.
\end{lemma}

\begin{proof}
Suppose that $x\leq f(x)$ and let $i\in [n]$. Then $x\leq f^i(x)$ so $f^i_i(x)=f_i(x)\leq f_i(f^i(x))$ and $f^i_j(x)=x_j\leq f_j(x)\leq f_j(f^i(x))$ for all $j\neq i$. Thus $x\leq f(x)$ implies $f^i(x)\leq f(f^i(x))$ for every $i\in [n]$. We deduce that, for any word $w=w_1w_2\dots w_k$, $x\leq f^{w_1}(x)\leq f^{w_1w_2}(x)\leq \dots \leq f^{w_1w_2\dots w_k}(x)$, and this clearly implies the lemma.
\end{proof}


\begin{proof}[Proof of Theorem~\ref{thm:monotone_universal}]
The proof is by induction on $n$. This is clear for $n = 1$, so suppose it holds for $n-1$. Fix an initial state $x\in\{0,1\}^n$ and, for every $z\in\{0,1\}^n$, let $z_{-n}:=(z_1, \dots, z_{n-1})$ and $h(z_{-n}):= f(z_{-n},x_n)_{-n}$. Then $h$ is a monotone network with $n-1$ components. Let $y:=f^{W^{n-1}}(x)$. Since the letter $n$ does not appear in $W^{n-1}$, we have $y_{-n}=h^{W^{n-1}}(x_{-n})$ and thus, by induction hypothesis, $y_{-n}$ is a fixed point of $h$. Hence, 
\[
f(y)=f(y_{-n},y_n)=(h(y_{-n}),f_n(y))=(y_{-n},f_n(y)). 
\]
We deduce that either $y$ is a fixed point of $f$, and in that case $y=f^{n,\omega^{n-1}}(y)=f^{W^n}(x)$ so we are done, or $f(y)=y+e_n$. Suppose that $f(y)=y+e_n$ with $y_n=0$, and remark that $f(y)=f^n(y)$. Setting $y':=f(y)$ we have $y\leq y'$, thus $y'\leq f(y')$, and since $y'_n=1$, we deduce that $\omega^{n-1}$ contains all the permutations of $\{i:y'_i=0\}$. Hence, according to Lemma~\ref{lm:local_increasing_or_decreasing} and Lemma~\ref{lem:increasing1}, 
\[
f^{\omega^{n-1}}(y')=f^{\omega^{n-1}}(f^n(y))=f^{n,\omega^{n-1}}(y)=f^{W^n}(x)
\]
is a fixed point of $f$. If $f(y)=y+e_n$ with $y_n=1$ the proof is similar.
\end{proof}

\section{Refinements and extensions}\label{sec:extensions}

\subsection{Conjunctive networks}

We now determine the maximum fixing length over all $n$-component conjunctive networks. Clearly the maximum is equal to one if $n=1$ and to two if $n=2$. To settle the case $n\geq 3$ and characterize the extremal networks, we need additional definitions. Let $C_n$ denote the $n$-vertex cycle (there is an edge from $i$ to $i+1$ for all $1\leq i<n$, and an edge from $n$ to $1$). We denote by $C^\circ_n$ the graph obtained from $C_n$ by adding an edge $(i,i)$ for all $i\in [n]$; these additional edges are called \BF{loops}. A strong component in a graph $G$ is \BF{initial} if there is no edge from a vertex outside the component to a vertex inside the component. 

\begin{theorem}
For all $n\geq 3$ and $f\in F_C(n)$, 
\[
\lambda(f)\leq  2n-2,
\]
with equality if and only if the interaction graph of $f$ is isomorphic to $C^\circ_n$.
\end{theorem}

\begin{proof}
Suppose that $n \ge 2$. Let $G$ be a graph on $[n]$, and let $f$ be the conjunctive network on $G$. A \BF{spanning in-tree} $S$ in $G$ rooted at $i$ is a spanning connected subgraph of $G$ such that all vertices $j \ne i$ have out-degree one in $S$, and $i$ has out-degree zero in $S$. A \BF{spanning out-tree} is defined similarly. It is clear that if $G$ is strong, then for any vertex $i$ there exists a spanning in-tree of $G$ rooted at $i$ (and similarly for out-trees). A vertex $l$ with in-degree zero in a spanning in-tree $S$ is referred to as a leaf of $S$. We denote the maximum number of leaves of a spanning in-tree of $G$ as $\phi(G)$. 
	
	We first prove the theorem when $G$ is strong. In that case, $f$ has exactly two fixed points: $\ZERO$ and~$\ONE$. 
	
	\begin{claim} \label{claim:h>2n-phi-1}
		If $G$ is strong, then $\lambda(f)\leq 2n - \phi(G) - 1$.
	\end{claim}
	
	\begin{proof}[Proof of Claim \ref{claim:h>2n-phi-1}]
Let $S$ be a spanning in-tree of $G$ with $\phi := \phi(G)$ leaves. Let $i_1i_2\dots i_n$ be a topological sort of $S$. The root of $S$ is thus $i_n$, and its leaves are $i_1\dots i_\phi$. Let $T$ be a spanning out-tree with the same root as $S$. Let $j_1j_2\dots j_n$ be a topological sort of $T$, so that its root is $j_1=i_n$. We claim that the word $w:=i_{\phi+1}\dots i_nj_2\dots j_n$ of length $2n - \phi - 1$ fixes $f$. Let $u:=i_{\phi+1}\dots i_n$ and $x\in\{0,1\}^n$. We set $x^\phi:=x$ and $x^k:=f^{i_k}(x^{k-1})$ for $\phi<k\leq n$. We claim that if $x^n_{i_n}=1$, then $f^u(x)=x^n=\ONE$. For otherwise, suppose $x^n_{i_n}=1$ and $x^n_{i_k}=0$ for some $\phi\leq k< n$. Let $i_{r_1}i_{r_2}\dots i_{r_p}$ be the path from $i_k=i_{r_1}$ to $i_{r_p}=i_n$ in $S$. This path follows the topological order, that is, $r_1< r_2<\dots < r_p$. Clearly, for all $1\leq q<p$ we have 
\[
x^n_{i_{r_q}}=0
~\Rightarrow~
x^{r_q}_{i_{r_q}}=0
~\Rightarrow~
x^{r_{q+1}-1}_{i_{r_q}}=0
~\Rightarrow~
x^{r_{q+1}}_{i_{r_{q+1}}}=0
~\Rightarrow~
x^n_{i_{r_{q+1}}}=0.
\]
Since $x^n_{i_k}=0$ we deduce that $x^n_{i_n}=0$, which is the desired contradiction. Hence, if $x^n_{i_n}=1$ then $f^u(x)=\ONE$ and thus $f^w(x)=\ONE$. Otherwise $x^n_{i_n}=0$ and it is easily shown by induction on $2\le k \le n$ that $f^{j_2\dots j_k}_{j_k}(x^n)=0$, thus $f^{j_2\dots j_n}(x^n)=f^w(x)=\ZERO$. 
	\end{proof}
	
We say that $G$ is a {\bf cycle with loops} if $G$ is isomorphic to a graph obtained from $C_n$ by adding some loops.

	\begin{claim} \label{claim:phi}
		If $G$ is strong and not a cycle with loops, then $\phi(G) \ge 2$ and hence $\lambda(f) \le 2n-3$.
	\end{claim} 
	
	\begin{proof}[Proof of Claim \ref{claim:phi}]
		Since adding loops to a graph maintains the value of $\phi$, without loss, suppose that $G$ has no loops. Since $G$ is strong but not a cycle, there exists a vertex $i$ in $G$ with in-degree $d \ge 2$. We can then construct a spanning in-tree rooted at $i$ with at least $d$ leaves as follows. For all $0 \le k<n$, let $U_k$ be the set of vertices $j$ such that $d_G(j,i)=k$ (i.e. $k$ is the minimum length of a path from $j$ to $i$ in $G$). Then $U_0$ only contains $i$, $U_1$ is the set of in-neighbors of $i$, and $U_0 \cup U_1 \cup \dots \cup U_{n-1} = [n]$. For any $j \in U_k$ with $1\leq k<n$, let $j'$ be any out-neighbor of $j$ in $U_{k-1}$. Then the edges $(j,j')$ for all $j\ne i$ form a spanning in-tree rooted at $i$ with at least $|U_1| = d\geq 2$ leaves.
	\end{proof}
	
	Suppose that $G$ is a cycle with loops, and let $L$ be the set of vertices with a loop. Given $l, l' \in L$, we say that $l'$ is the successor of $l$ if none of the internal vertices on the path from $l$ to $l'$ belong to $L$. The maximum distance in $G$ from a vertex in $L$ to its successor is denoted as $d(G)$. By convention, we let $d(G):= n$ if $|L| = 0$ or $|L| = 1$.
	
	\begin{claim} \label{claim:cycle_loops_upper}
		If $G$ is a cycle with loops, then $\lambda(f) \le 2n - d(G) - 1$. Therefore, if $G$ is not isomorphic to $C^\circ_n$, then $\lambda(f)\le 2n-3$.
	\end{claim}
	
	\begin{proof}[Proof of Claim \ref{claim:cycle_loops_upper}]
Without loss, we assume that $G$ is obtained from $C_n$ by adding some loops. Let us first settle the case where $|L| \le 1$. If $L$ is empty, then it is easy to see that $1,2,\dots ,n-1$ fixes $f$. If $L$ is a singleton we may assume, without loss, that $n$ is the only vertex with a loop, and then the same strategy works: $1,2,\dots ,n-1$ fixes $f$. Henceforth, we assume $|L| \ge 2$.
		Without loss, suppose that $n$ and $d:=d(G)$ both belong to $L$ and that $d$ is the successor of $n$. Then we claim that the word $w:= d+1, d+2, \dots,n,1,\dots, d, d+1, \dots, n-1$ fixes $f$. Let $x\in\{0,1\}^n$. Firstly, suppose $x_l = 0$ for some $l \in L$; we note that $d\leq l\leq n$. First of all, the value of $x_l$ will remain zero: $f^{d+1, \dots, l}_l(x) = 0$. Afterwards, the $0$ will propagate through the cycle: $f^{d+1, \dots, l, \dots, l-1}(x) = \ZERO$. Secondly, if $x_l = 1$ for all $l \in L$, then it is easy to show that $f^w(x)=\ONE$.
\end{proof}

The two previous claims show that if $G$ is strong then $\lambda(f)\leq 2n-2$, with a strict inequality when $G$ is not isomorphic to $C^\circ_n$. The lower bound below thus settles the strong case.

\begin{claim} \label{claim:cycle_loops_lower}
		If $G$ is isomorphic to $C^\circ_n$ then $\lambda(f)\ge 2n-2$.
\end{claim}
	
\begin{proof}[Proof of Claim \ref{claim:cycle_loops_lower}]
For all $1 \le i \le n$ and $0\leq k<n$, we denote by $i_k$ the vertex at distance $k$ from $i$ in $G$, and we denote by $x^{i,k}$ the state such that $x^{i,k}_j=0$ if and only if the distance between $i$ and $j$ is at most $k$. Thus $x^{i,0}=\ONE+e_i$ and $x^{i,n-1}=\ZERO$. Furthermore, for all $0\leq k<n-1$, 
\[
f(x^{i,k})=x^{i,k}+e_{i_{k+1}}=x^{i,k+1}.
\]
We deduce that if $w=w_1w_2\dots w_p$ fixes $f$, then $f^w(x^{i,0})=\ZERO$ and, necessarily, $i_1i_2\dots i_{n-1}$ is a subsequence of $w$ for all $i$. Let $i$ be the last index to appear in $w$, then $i=w_q$ for some $q \ge n$; then the word $i_1i_2\dots i_{n-1}$ begins in position $q$ of $w$ and does not end before position $q+n-2 \ge 2n-2$. Hence $\lambda(f)\ge 2n-2$. 
\end{proof}

It remains to settle the non-strong case. We first establish an upper-bound on $\lambda(f)$ that depends on the decomposition of $G$ in strong components. Let $\psi_1(G)$ be the number of initial strong components containing a single vertex without a loop, let $\psi_2(G)$ be the number of initial strong components containing a single vertex with a loop, let $\psi_3(G)$ be the number of initial strong components with at least two vertices, and let $\psi_4(G)$ be the number of non-initial strong components. 

	\begin{claim} \label{claim:psi}
$\lambda(f)\leq 2n-\psi_1(G)-2\psi_2(G)-2\psi_3(G)-\psi_4(G)$.
	\end{claim}
	
\begin{proof}[Proof of Claim \ref{claim:psi}]
Let $I_1, \dots, I_k$ denote the strong components of $G$ in the topological order, and let $n^l:=|I_l|$. We then consider a word $w^l$ that fixes the conjunctive network on $G[I_l]$.
		\begin{enumerate}
			\item If $I_l$ is an initial strong component containing a single vertex $i$ without a loop, then  $w^l:= i$; $w^l$ has length $2n^l - 1$.
			
			\item If $I_l$ is an initial strong component containing a single vertex with a loop, then $w^l$ is the empty word; $w^l$ has length $2 n^l - 2$.
			
			\item If $I_l$ is an initial strong component with at least two vertices we consider two cases. If $G[I_l]$ is not a cycle with loops, then $w^l$ is the word described in the proof of Claim \ref{claim:h>2n-phi-1}. If $G[I_l]$ is a cycle with loops, then $w^l$ is the word described in the proof of Claim \ref{claim:cycle_loops_upper}. In both cases, $w^l$ has length at most $2n^l - 2$.
			
			\item Otherwise, $I_l$ is a non-initial strong component. Let $S$ be a spanning in-tree of $G[I_l]$ and let $i_1i_2\dots i_n$ be a topological sort of $S$. The root of $S$ is thus $i_n$. Let $T$ be a spanning out-tree with the same root, and let $j_1j_2\dots j_n$ be a topological sort of $T$, so that $j_1=i_n$. Then $w^l:=i_1\dots i_nj_2\dots j_n$; $w^l$ has length $2n^l - 1$.
		\end{enumerate}
		Then, by induction on $l$, it is easily proved that $f^{w^1w^2 \dots w^l}$ fixes the conjunctive network on the subgraph of $G$ induced by $I_1\cup I_2\cup \dots\cup I_l$. Thus $w := w^1w^2 \dots w^k$ fixes $f$ and has length at most $2n-2\psi_2(G)-2\psi_3(G)-\psi_1(G)-\psi_4(G)$. 
	\end{proof}

We can finally prove that $\lambda(f) \le 2n-3$ if $G$ is not strong. 
	
	
	\begin{claim} \label{claim:problematic}
		If $G$ is not strong and $n\geq 3$, then $\lambda(f) \le 2n-3$.
	\end{claim}
	
	\begin{proof}[Proof of Claim \ref{claim:problematic}]
Suppose first that $\psi_4(G)=0$ (then $G$ is the disjoint union of strong graphs). If $\psi_2(G)+\psi_3(G)\geq 2$ then $\lambda(f)\leq 2n-4$, and if $\psi_2(G)+\psi_3(G)=1$ then $\psi_1(G)\geq 1$, since $G$ is not strong, and thus $\lambda(f)\leq n-3$. Finally, if $\psi_2(G)+\psi_3(G)=0$ then $\psi_1(G)\geq 3$ since $n\geq 3$, and thus  $\lambda(f)\leq n-3$. Suppose now that $\psi_4(G)\geq 1$. If $\psi_1(G)\geq 2$ or $\psi_2(G)\geq 1$ or $\psi_3(G)\geq 1$ or $\psi_4(G)\geq 2$ then $\lambda(f)\leq 2n-3$. So assume that $\psi_1(G)=\psi_4(G)=1$ and $\psi_2(G)=\psi_3(G)=0$. This means that $G$ is connected, has a unique initial strong component containing a single vertex without a loop, and has a unique non-initial strong component, with at least two vertices, since $n\geq 3$. Suppose, without loss, that $n$ is the vertex of the initial strong component, and let $x\in\B^n$. Since $f_n$ is the empty conjunction, we have $f_n(x) = 1$. Let $h$ be the conjunctive network on the (strong) graph $H$ obtained from $G$ by removing vertex $n$. Then for any word $u$ we have $f^{n,u}(x)=(h^u(x_1,\dots,x_{n-1}),1)$. Thus, let $u$ be the word of length at most $2(n-1) - 2$ fixing $h$ from the proof of Claim \ref{claim:h>2n-phi-1} (if $H$ is not a cycle with loops) or Claim \ref{claim:cycle_loops_upper} (otherwise). Then $w= n,u$ is a word of length at most $2n-3$ fixing $f$.  
\end{proof}
This completes the proof of the theorem.	\end{proof}

\begin{remark}
We can strengthen the upper bound for specific graphs. In particular, if $G$ is undirected and connected, then there are lower bounds on the maximum number of leaves of a spanning tree for $G$ (see \cite{BK12} for instance).
\end{remark}

\subsection{Monotone networks with a given interaction graph}

We now refine Theorem \ref{thm:monotone_universal} for $F_M(G)$, the family of monotone networks whose interaction graph is contained in a graph $G$. Recall that the \BF{transversal number} of $G$ is the minimum size of a subset $I$ of vertices in $G$ such that $G\setminus I$ is acyclic. The main result is that, for fixed transversal number, the fixing length of $F_M(G)$ is linear in the number of vertices. The statement needs additional definitions. 

\begin{definition}[$(i,\alpha)$-complete words]
For $i\geq 0$ and $\alpha\geq 0$, a word on $[\alpha+i]$ is \BF{$(i,\alpha)$-complete} if it contains, as subsequences, all the permutations $j_1,\dots,j_{\alpha+i}$ of $[\alpha+i]$ such that, for all $1\leq \ell< \alpha+i$, if $j_{\ell},j_{\ell+1}\in [\alpha]$ then $j_\ell<j_{\ell+1}$. We denote by $\lambda(i,\alpha)$ the length of a shortest $(i,\alpha)$-complete word.
\end{definition}

Thus $\lambda(0,\alpha)=\alpha$. Furthermore, for $i>0$, we have $\lambda(i,\alpha)\leq \lambda(\alpha+i)$, with equality if and only if $\alpha\in\B$. In a graph $G$, a \BF{$1$-feedback vertex set} is a set of vertices $I$ such that all the cycles of $G\setminus I$ are loops (i.e. cycles of length one). The \BF{$1$-transversal number} of $G$ is the minimum size of a $1$-feedback vertex set. Clearly, if $G$ is an $n$-vertex graph with transversal number $\tau$ and $1$-transversal number $\tau_1$, then $\tau_1\leq\tau$ and $\tau_1<n$. 

The following is a quantitative version of Theorem~\ref{thm:main4} stated in the introduction. 

\begin{theorem}\label{thm:tau_1}
Let $G$ be a graph on $[n]$ with $1$-transversal number $\tau_1$. We have 
\[
	\lambda_M(G)\leq n+\sum_{i=1}^{\tau_1} \lambda(i-1,n-\tau_1)\leq \left(\frac{\tau_1^2}{2}+\frac{3\tau_1}{2}+1\right)n.
\]
\end{theorem}

\begin{remark}
Let $K_n$ be the complete directed graph on $[n]$ (with $n^2$ edges). Since the $1$-transversal number of $K_n$ is $n-1$, we have the following, which proves that Theorem~\ref{thm:tau_1} indeed contains Theorem~\ref{thm:monotone_universal}: 
\[
\lambda_M(n)=\lambda_M(K_n)\leq n+\sum_{i=1}^{n-1} \lambda(i-1,1)=n+\sum_{i=1}^{n-1} \lambda(i).
\]
\end{remark}

\begin{proof}[Proof of Theorem~\ref{thm:tau_1}]
Let $G$ be a graph on $[n]$ with $1$-transversal number $\tau_1$ and let $\alpha=n-\tau_1$. Let $]\alpha,n]=\{\alpha+1,\dots,n\}$. Without loss, we assume that $]\alpha,n]$ is a $1$-feedback vertex set. We also assume that $12\dots\alpha$ is the topological order of $G[\{1,\dots,\alpha\}]$; this order exists, since all the cycles of $G[\{1,\dots,\alpha\}]$ have length one.

For all $1\leq i\leq n$, let $R_i$ be the set of vertices reachable from $i$ in $G[\{1,\dots,i\}]$. Thus $R_i=\{i\}$ if $i\leq\alpha$, and $R_i\subseteq [i]$ otherwise. Let $P_i$ be the set of enumerations $j_1j_2\dots j_k$ of $R_i\setminus\{i\}$ such that, for all $1\leq \ell<k$, if $j_\ell,j_{\ell+1}\in [\alpha]$ then $j_\ell<j_{\ell+1}$. Let $\omega^i$ be a shortest word containing, as subsequences, all the enumerations contained in $P_i$. Let $w^i:=i,\omega^i$ and 
\[
	W:=w^1,\dots,w^n.
\]

If $i\in [\alpha]$, then $R_i=\{i\}$, thus $\omega^i=\epsilon$. Furthermore, if $i\in ]\alpha,n]$, then $R_i\subseteq [i]$ and we deduce that $|\omega^i|\leq \lambda(i-\alpha-1,\alpha)$. Thus 
\[
	|W|\leq n+\sum_{i=1}^{\tau_1} \lambda(i-1,n-\tau_1).
\]
Let us now prove that $W$ fixes $F_M(G)$. For all $i\in [n]$, let 
\[
	W^i:=w^1,\dots,w^i
	\quad\text{and}\quad
	G_i:=G[\{1,\dots,i\}].
\]
We prove, by induction on $i$, that $W^i$ fixes $F_M(G_i)$. This is obvious for $i=1$. Assume that $i\geq 2$. Let $f\in F_M(G_i)$ and $x\in\B^i$. We write $x=(x_{-i}, x_i)$ and set 
\[
	f'(x_{-i}) := f(x_{-i}, x_i)_{-i}.
\]
In this way, $f'\in F_M(G_{i-1})$. Let 
\[
	y := f^{W^i}(x).
\]

Since $y_{-i} = f'^{W^{i-1}}(x_{-i})$, by induction hypothesis, $y_{-i}$ is a fixed point of~$f'$. We deduce that either $y$ is a fixed point of $f$, and in that case 
\[
	f^{W^i}(x)=f^{w^i}(f^{W^{i-1}}(x))=f^{w^i}(y)=y
\]
is a fixed point of $f$, and we are done, or $f(y)=y+e_i$.

So it remains to suppose that $f(y)=y+e_i$ and to prove that $f^{w^i}(y)$ is a fixed point. We consider the case where $y\leq f(y)$, the other case being similar. Let 
\[
	y^0:=y
	\quad\text{and}\quad
	y^k:=f^{w^i_1w^i_2\dots w^i_k}(y)
\] 
for all $1\leq k \leq p$, with $p=|w^i|$. According to Lemma~\ref{lm:local_increasing_or_decreasing} we have 
\[
	y^0\leq y^1\leq\dots\leq y^k\leq f(y^k). 
\]

Let us prove that $y^p=f^{w^i}(y)$ is a fixed point of $f$. Let $j_1j_2 \dots j_d$ be the ordered sequence of coordinates that turned from $0$ to $1$ during the sequence $y^0,y^1,\dots,y^p$. In this way, $d$ is the Hamming distance between $y^0$ and $y^p$, and $j_1=i=w^i_1$. Furthermore, 
\[
	f^{j_1j_2\dots j_d}(y) = y^p.
\]

Suppose, for the sake of contradiction, that $f_j(y^p) \neq y^p_j$ for some $1\leq j\leq i$. Since $y^p\leq f(y^p)$, we must have 
\[
y^p_j<f_j(y^p).
\]
Thus $y^k_j=0$ for all $0\leq k\leq p$. Hence, $j$ does not appear in the sequence $j_1 j_2\dots j_d$. Let
\[
j_{d+1}:=j,
\]
and let us prove that 
\begin{equation}\label{eq:R_i}
	\{j_1,\dots,j_d,j_{d+1}\}\subseteq R_i.
\end{equation}
Since $j_1=i$ we have $j_1\in R_i$. We now prove $j_k\in R_i$ with $k\neq 1$. Let $y^q$ be the smallest index $0\leq q\leq p$ such that $y^q_{j_k}<f_{j_k}(y^q)$. Since $k\neq 1$, $j_k\neq i$, and since $f(y)=y+e_i$, we deduce that $q>1$. Then, by the choice of $q$, we have $y^{q-1}_{j_k}=f_{j_k}(y^{q-1})$ and thus $y^{q-1}_{j_k}=y^q_{j_k}$. Hence, $f_{j_k}(y^{q-1})<f_{j_k}(y^q)$. Thus $G$ has an edge from $w_{q-1}$ to $j_k$, since $w_{q-1}$ is the unique component that differs between $y^{q-1}$ and $y^q$. Clearly, $w_{q-1}=j_\ell$ for some $1\leq \ell<k$. Thus, we have proved that for all $j_k$ with $1<k\leq d+1$, there exists $1\leq \ell <k$ such that $j_\ell j_k$ is an edge of $G$. We deduce that all the $j_k$ with $1<k\leq d+1$ are reachable from $j_1=i$. This proves \eqref{eq:R_i}. 

Furthermore, for all $1\leq \ell\leq d$, if $j_\ell,j_{\ell+1}\in [\alpha]$ and $j_\ell>j_{\ell+1}$, then 
\[
f^{j_1j_2\dots j_\ell j_{\ell+1}\dots j_{d+1}}(y)\leq f^{j_1j_2\dots j_{\ell+1}j_\ell \dots j_{d+1}}(y)
\]
since $G$ has no edge from $j_\ell$ to $j_{\ell+1}$. Thus, by applying such switches several times, we can reorder the sequence $j_1 j_2\dots j_{d+1}$ into a sequence $s_1s_2\dots s_{d+1}$ such that
\[
f^{s_1 s_2\dots s_d s_{d+1}}(y)\geq f^{j_1j_2\dots j_dj_{d+1}}(y)
\]
and such that, for all $1\leq \ell\leq d$, if $s_\ell,s_{\ell+1}\in [\alpha]$ then $s_\ell<s_{\ell+1}$. In this way, $s_1=j_1=i$, and $s_2,\dots,s_{d+1}$ is in $P_i$. Hence, by definition, $s_2\dots s_{d+1}$ is a subsequence of $\omega^i$, and thus $s_1s_2\dots s_{d+1}$ is a subsequence of $w^i$. Therefore,  
\[
	y^p_j=y^p_{j_{d+1}}= f_{j_{d+1}}^{\omega^i}(y) \geq f_{j_{d+1}}^{s_1s_2\dots s_d s_{d+1}}(y) \geq  f_{j_{d+1}}^{j_1j_2\dots j_d j_{d+1}}(y) = f_{j_{d+1}}(y^p)=f_j(x^p),
\]
a contradiction. Thus $W^i$ fixes $f$, and thus the whole family $F_M(G_i)$. 

Therefore, $W$ fixes $F_M(G)$ and it remains to prove that $|W|\leq (\frac{\tau_1^2}{2}+\frac{3\tau_1}{2}+1)n$. This follows from the proposition below and an easy computation. 
\end{proof}

\begin{proposition} 
For all $i\geq 0$ and $\alpha\geq 0$ we have $\lambda(i,\alpha)\leq i^2+i\alpha+\alpha$.
\end{proposition}

\begin{proof}
Let $\beta:=\alpha+i$ and consider the word $w:=i\cdot (12\dots \beta),12 \dots\alpha$, resulting from the concatenation of $i$ copies of $12\dots\beta$ and the addition of the suffix $12\dots\alpha$. Let $u=j_1j_2\dots j_{\beta}$ be a permutation of $[\beta]$ such that, for all $1\leq \ell<\beta$, if $j_\ell,j_{\ell+1}\in [\alpha]$ then $j_\ell<j_{\ell+1}$. We will prove that $w$ is $(i,\alpha)$-complete and, for that, it is sufficient to prove that $u$ is contained in $w$. Let $j_{k_1}\dots j_{k_i}$ be the longuest subsequence of $u$ with letters in $[\beta]\setminus [\alpha]$. Then,
\[
	\begin{array}{rcl}
	j_1\dots j_{k_1}&\text{is a subsequence of}&(1\dots\beta)=w_1\dots w_{\beta}\\
	j_{k_1+1}\dots j_{k_2}&\text{is a subsequence of}&(1\dots \beta)=w_{\beta+1}\dots w_{2\beta},\\
	&\vdots&\\
	j_{k_{i-1}+1} \dots j_{k_i}&\text{is a subsequence of}&(1\dots \beta)=w_{(i-1)\beta+1}\dots w_{i\beta},\text{ and}\\
	j_{k_i+1}\dots j_\beta&\text{is a subsequence of}&(1\dots \alpha)=w_{i\beta+1}\dots w_{i\beta+\alpha}.\\
	\end{array}
\]
Thus $u$ is a subsequence of $w$. Since $|w|=i^2+i\alpha+\alpha$, this proves the proposition. 
\end{proof}

\subsection{Balanced networks}

We now consider a family of networks (namely, balanced networks) which is more general than monotone networks. Those are defined by their signed interaction graph, hence we review basic definitions and properties of signed graphs first.

A \BF{signed graph} is a couple $(G,\sigma)$ where $G$ is a graph, and $\sigma:E\to\{-1,0,1\}$ is an edge labelling function, that gives a (positive, negative or null) sign to each edge of $G$. The \BF{sign} of a cycle in $(G,\sigma)$ is the product of the signs of its edges, and $(G,\sigma)$ is \BF{balanced} if all the cycles are positive. The \BF{signed interaction graph} of an $n$-component network $f$ is the signed graph $(G,\sigma)$ where $G$ is the interaction graph of $f$ and where $\sigma$ is defined for each edge of $G$ from $j$ to $i$ as follows: 
\[
\sigma(ji):=
\left\{
\begin{array}{rl}
1&\text{if }f_i(x)\leq f_i(x+e_j)\text{ for all }x\in\B^n\text{ with }x_j=0,\\
-1&\text{if }f_i(x)\geq f_i(x+e_j)\text{ for all }x\in\B^n\text{ with }x_j=0,\\
0&\text{otherwise.}
\end{array}
\right.
\]

\begin{definition}[Balanced networks]
An network is \BF{balanced} if its signed interaction graph is balanced. The family of $n$-component balanced networks is denoted $F_B(n)$. 
\end{definition}

Clearly, a network is monotone if and only if all the edges of its signed interaction graph are positive. Thus every monotone network is balanced. Conversely, a balanced network can be ``decomposed'' into monotone networks by considering the decomposition of its interaction graph into strong components, as formally described below. 

Given $z\in\B^n$, the \BF{$z$-switch} of $f$ is the $n$-component network $f'$ defined by 
\[
f'(x)=f(x+z)+z
\]
for all $x\in\B^n$. For instance, the $\ONE$-switch of $f$ is the dual of $f$. If $f'$ is the $z$-switch of $f$, then $f$ and $f'$ have the same interaction graph $G$, but their signed interaction graph $(G,\sigma)$ and $(G,\sigma')$ may differ, since $\sigma'(ji)=\sigma(ji)$ for all edge $ji$ with $z_j=z_i$ but $\sigma'(ji)=-\sigma(ji)$ for all edge $ji$ with $z_j\neq z_i$. Clearly, if $f'$ is the $z$-switch of $f$, then $f$ is the $z$-switch of $f'$, and we then say that $f$ and $f'$ are \BF{switch-equivalent}. 

\begin{proposition}[\cite{MRRS13}]
Let $f$ be a network with a strong interaction graph. Then $f$ is balanced if and only if $f$ is switch-equivalent to a monotone network.
\end{proposition}

The proposition above have immediate consequences on the existence of short words fixing the family of balanced networks. Clearly, if $f$ and $f'$ are switch-equivalent, then any word fixing $f$ fixes $f'$ as well. Therefore, let $W^n$ be a word fixing $F_M(n)$ and consider $n \cdot W^n$ (the word $W^n$ repeated $n$ times). Let $f \in F_B(n)$ and denote the strong components of its interaction graph as $I_1, \dots, I_k$ ($k \le n$). Since $f$ restricted to each strong component is switch-equivalent to a monotone network, $W^n$ fixes each strong component individually, and thus $l\cdot W^n$ fixes the first $l$ strong components. In particular, $n \cdot W^n$ fixes $f$. Thus, by Theorem \ref{thm:monotone_universal}, there exists (for sufficiently large $n$) a word of length at most $n^4/3$ fixing $F_B(n)$.

The following theorem refines (and gives a formal proof of) the result above. More precisely, let $n = 3q + r$ with $0 \le r < 3$, let $s$ be the word $s := 12 \dots n$ and let $W^n$ be any word fixing $F_M(n)$ of minimal length. Then define the word 
\[
\tilde W^n := q \cdot (s s W^n), r \cdot s
\]
of length $q(2n+\lambda_M(n))+rn$. 

\begin{theorem} \label{thm:balanced_universal}
The word $\tilde W^n$ fixes $F_B(n)$ for every $n\geq 1$. Therefore, 
	$$
	\lambda_M(n) \le \lambda_B(n) \le \frac{n}{3} \lambda_M(n) + n^2.
	$$
\end{theorem}

\begin{proof}
Let $n=3q+r$, with $0\leq r<3$, and let $X^n$ be a word fixing $F_M(n)$. We prove, more generally, that $\tilde X^n:= (q \cdot (s s X^n), r \cdot s)$ fixes $F_B(n)$. Let $f \in F_B(n)$ and let $G$ be the interaction graph of $f$. 

The main idea of the proof is that each factor $w:=ssX^n$ of $\tilde X^n$ fixes at least three new vertices of $G$. Therefore, $q \cdot w$ fixes at least $3q = n-r$ vertices, and finally $r \cdot s$ fixes the last $r$ vertices if need be. 

We formally proceed by induction on $n$. If $n=1$ then $s=1$ fixes $f$, and if $n=2$, it is easy to check that $ss=1212$ fixes $f$. So we assume that $n\geq 3$. We say that a prefix $u$ of $\tilde W^n$ fixes a set of vertices $I\subseteq [n]$ if, for any other prefix $v$ longer than $u$, we have  $f^u_i(x)= f^v_i(x)$ for all $i\in I$. We consider three cases, and in each case, we select a subset $I$ of vertices of size at least three fixed by $w$. 

\begin{enumerate}
\item
{\em $G$ has an initial strong component $I$ with at least three vertices.} Then let $I$ be this initial strong component, and let $g$ be the restriction of $f$ on $I$. Since $g$ is switch-equivalent to a monotone network, $X^n$ fixes $g$, and thus $w$ fixes $I$.  

\item 
{\em $G$ has an initial strong component with two vertices, say $I_1 = \{i,j\}$ with $i<j$.} Again, let $g$ be the restriction of $f$ on $I_1$. The occurrences of $i$ and $j$ in $ss = 12 \dots  n12 \dots, n$ are $ijij$, in that order; this contains $iji$, which fixes $g$. Therefore, $ss$ fixes $I_1$. Suppose, without loss, that $i=n-1$ and $j=n$, and let $h$ be the $(n-2)$-component network defined by 
\[
h(y):=(f_1(y,z)),\dots,f_{n-2}(y,z))
\quad\text{with}\quad
z:=(f^{ss}_{n-1}(x),f^{ss}_n(x))
\]
for all $y\in\B^{n-2}$. Then $h$ is balanced and, by a reasoning similar to the first case, $X^n$ fixes an initial strong component $I_2$ of the interaction graph of $h$. Thus, $w$ fixes $I:=I_1 \cup I_2$.

\item 
{\em All the initial strong components of $G$ have one vertex each.} Note that $s$ fixes all the initial strong components. Therefore, if there are three initial strong components $\{i_1\},\{i_2\},\{i_3\}$, then $s$ fixes $I:=\{i_1,i_2,i_3\}$ and we are done. If there are two initial strong components $\{i_1\},\{i_2\}$ then $s$ fixes $I_1:=\{i_1,i_2\}$ and again $X^n$ fixes a non-empty subset $I_2$ of vertices, as shown in the second case. Thus $w$ fixes $I:=I_1\cup I_2$. There is only one case left: $I_1 = \{i_1\}$ is the only initial strong component. We then consider an initial strong component $I_2$ of $G\setminus I_1$. If $|I_2| \ge 2$, then $s$ fixes $I_1$ and $X^n$ fixes $I_2$. Thus $w$ fixes $I:=I_1\cup I_2$ and we are done. If $I_2=\{i_2\}$, then $ss$ fixes $I_1\cup I_2$, and again $X^n$ fixes a non-empty subset of vertices $I_3$. Thus $w$ fixes $I:=I_1\cup I_2\cup I_3$ and we are done.
\end{enumerate}

Thus, in any case, there exists a subset $I$ of three vertices fixed by $w$. Suppose, without loss, that $I=\{n-2,n-1,n\}$. Then, let $h$ be the $(n-3)$-component network defined  by 
\[
h(y):=(f_1(y,z)),\dots,f_{n-3}(y,z))
\quad\text{with}\quad
z:=(f^w_{n-2}(x),f^w_{n-1}(x),f^w_n(x))
\]
for all $y\in\B^{n-3}$. Then $h$ is balanced, and thus, by induction, $\tilde X^{n-3}$ fixes $h$. Consequently, $w$ fixes $I$, and then $\tilde X^{n-3}$ fixes $[n]\setminus I$. Since $\tilde X^n=w,\tilde X^{n-3}$, we deduce that $\tilde X^n$ fixes $f$.
\end{proof}

\section{Conclusion}\label{sec:conclusion}

In this paper, we have considered the asynchronous automaton associated with a Boolean network and used it to introduce the family of fixable networks (which is huge by Theorem~\ref{thm:huge}). We have then introduced the fixing length $\lambda(f)$ of a fixable network $f$ and the fixing length $\lambda(\mathcal{F})$ of a family $\mathcal{F}$ of fixable networks. We have then identified several families with polynomial fixing lengths, using properties concerning complete words. Our results are summarised in Table \ref{table:summary}. A dash means that we did not find any nontrivial result for the given entry. 

Our main results concern the family $F_M(n)$ of $n$-component monotone networks. In particular, we have proved that the fixing length of $F_M(n)$ is at most cubic and that the maximum fixing length of network in $F_M(n)$ is at least quadratic. The main open question raised by these results is the following: {\em is there an asymptotic gap between the fixing length of $F_M(n)$ and the maximum fixing length of network in $F_M(n)$?} A positive answer is not obvious since, for instance, for the family $F_I(n)$ of $n$-component increasing networks, which is doubly exponential in $n$, we have proved the following: both the fixing length of $F_I(n)$ and the maximum fixing length of network in $F_I(n)$ are quadratic. 

There are some connections between fixability and synchronization, since a network with a unique fixed point is fixable if and only if its asynchronous automaton is synchronizing. It would be interesting to study synchronization in Boolean networks more specifically. In particular, it would be interesting to study the famous \v{C}ern\'y conjecture, stated in the general framework of deterministic finite automata, in the specific setting of Boolean networks, that is, for the class of asynchronous automata associated with Boolean networks. 

\begin{table}
\renewcommand{\arraystretch}{1.5}
\centering
\begin{tabular}{|l|l||l|l|}
	\hline
	Networks		& $\mathcal{F}$			& $\max_{f\in\mathcal{F}} \lambda(f)$	& $\lambda(\mathcal{F})$	\\
	\hline
	Acyclic 		& $F_A(n)$				& $= n$ 								& $= \lambda(n)\sim n^2$			\\
	Path			& $F_P(n)$				& $= n$									& $= \lambda(n)\sim n^2$ 			\\
	Increasing 		& $F_I(n)$				& $\ge (\frac{1}{e}-\varepsilon)n^2$	& $= \lambda(n)\sim n^2$ 			\\
	Monotone		& $F_M(n)$				& $\ge (\frac{1}{e}-\varepsilon)n^2$	& $\le \frac{1}{3} n^3$		\\
	Conjunctive		& $F_C(n)$				& $=2n-2$ 								& --						\\
	$G$-monotone	& $F_M(G)$				& --									& $\le 2\tau^2n+n$			\\
	Balanced		& $F_B(n)$				& --									& $\le \frac{1}{9} n^4$		\\
	\hline
\end{tabular}
\caption{Summary of results} \label{table:summary}
\end{table}

\paragraph{Acknowledgment} 
This work is partially supported by: FONDECYT project 1151265, Chile; CONICYT /  PIA project  AFB 170001, Chile; Labex UCN@Sophia, Universit\'e C\^ote d'Azur, France; CNRS project PICS06718; project STIC AmSud CoDANet 19-STIC-03 (Campus France 43478PD), France; and Young Researcher project ANR-18-CE40-0002-01 ``FANs'', France.

\bibliographystyle{plain}
\bibliography{BIB}

\end{document}